\documentclass[a4paper,reqno]{amsart}

\usepackage{graphicx}
\usepackage{amsmath}
\usepackage{amssymb}
\input xy
\xyoption{all}

\usepackage{color}

\numberwithin{equation}{section}

\theoremstyle{plain}

\newtheorem{thm}{Theorem}[section]
\newtheorem{cor}[thm]{Corollary}
\newtheorem{lem}[thm]{Lemma}
\newtheorem{prop}[thm]{Proposition}

\newtheorem{defn}[thm]{Definition}
\newtheorem{exm}[thm]{Example}

\theoremstyle{remark}
\newtheorem{rem}[thm]{Remark}

\newdir{ >}{{}*!/-5pt/\dir{>}}

\renewcommand{\mod}{\operatorname{mod}\nolimits}

\newcommand{\Hom}{\operatorname{Hom}\nolimits}

\newcommand{\Ext}{\operatorname{Ext}\nolimits}

\newcommand{\pd}{\operatorname{pd}\nolimits}
\newcommand{\id}{\operatorname{id}\nolimits}
\newcommand{\gl}{\operatorname{gl.dim}\nolimits}
\renewcommand{\sup}{\operatorname{sup}\nolimits}

\newcommand{\Cone}{\operatorname{Cone}\nolimits}
\newcommand{\CoCone}{\operatorname{CoCone}\nolimits}

\newcommand{\M}{\mathcal M}

\newcommand{\B}{\mathcal B}
\newcommand{\uB}{\underline{\B}}
\newcommand{\oB}{\overline{\B}}
\newcommand{\U}{\mathcal U}
\newcommand{\V}{\mathcal V}

\newcommand{\W}{\mathcal W}
\newcommand{\h}{\mathcal H}

\newcommand{\D}{\mathcal D}

\newcommand{\K}{\mathcal K}
\newcommand{\N}{\mathcal N}
\newcommand{\R}{\mathcal R}
\newcommand{\X}{\mathcal X}

\newcommand{\C}{\mathcal C}

\newcommand{\svecv}[2]{\left(\begin{smallmatrix}
      #1 \\
      #2
    \end{smallmatrix}\right)}

\newcommand{\svech}[2]{\left(\begin{smallmatrix}
      #1 & #2
\end{smallmatrix}\right)}

\renewcommand{\emph}{\textit}
\renewcommand{\phi}{\varphi}

\begin{document}

\title{Abelian quotients associated with fully rigid subcategories}

\author{Yu Liu}
\address{School of Mathematics, Southwest Jiaotong University, 610031, Chengdu, Sichuan, People's Republic of China}
\email{liuyu86@swjtu.edu.cn}

\begin{abstract}
In this article, we study the Gorenstein property of abelian quotient categories induced by fully rigid subcategories on an exact category $\B$. We also study when $d$-cluster tilting subcategories become fully rigid. We show that the quotient abelian category induced by such $d$-cluster tilting subcategories are hereditary.
\end{abstract}

\maketitle

\section{Introduction}

Koenig and Zhu \cite{KZ} showed that any ideal quotient of a triangulated category moduloing a cluster tilting subcategory is an abelian category and this abelian quotient category is the module category of a $1$-Gorenstein algebra. It generalized a work of Keller and Reiten \cite{KR} for 2-Calabi-Yau triangulated categories. Nakaoka \cite{N1} introduced the notion of cotorsion pairs in triangulated categories and showed that from any cotorsion pair, one can construct an abelian category, which agrees with Koenig and Zhu's abelian quotient category when the cotorsion pair comes from a cluster tilting subcategory. Beligiannis \cite{B} investigated cotorsion pairs arising from some subcategories called fully rigid, and extended the above result concerning the Gorenstein property.


Both triangulated and exact category are important structures for representation theory. Many results for cotorsion pairs are shown to be similar on triangulated and exact category, it is reasonable to consider the Gorenstein property for the similar structure on exact categories. We briefly review the important properties of exact category. For more details, we refer to \cite{Bu}. Let $\mathcal A$ be an additive category, we call a pair of morphisms $(i,d)$ a \emph{weak short exact sequence} if $i$ is the kernel of $d$ and $d$ is the cokernel of $i$. Let $\mathcal E$ be a class of weak short exact sequences of $\mathcal A$, stable under isomorphisms, direct sums and direct summands. If a weak short exact sequence $(i,d)$ is in $\mathcal E$, we call it a \emph{short exact sequence} and denote it by
$$\xymatrix{X \;\ar@{>->}[r]^i &Y \ar@{->>}[r]^d &Z.}$$
We call $i$ an \emph{inflation} and $d$ a \emph{deflation}. The pair $(\mathcal A,\mathcal E)$ (or simply $\mathcal A$) is said to be an \emph{exact category} if it satisfies the following properties:

\begin{itemize}
\item[(a)] Identity morphisms are inflations and deflations.

\item[(b)] The composition of two inflations (resp. deflations) is an inflation (resp. deflation).

\item[(c)] If $\xymatrix{X \;\ar@{>->}[r]^i &Y \ar@{->>}[r]^d &Z}$ is a short exact sequence, for any morphisms $f : Z' \rightarrow Z$ and $g: X \rightarrow X'$, there are commutative diagrams
$$\xymatrix{
Y' \ar[d]_{f'} \ar@{->>}[r]^{d'} \ar@{}[dr]|{PB} & Z' \ar[d]^f\\
Y \ar@{->>}[r]_d & Z\\
} \quad \xymatrix{
X \ar[d]_{g} \ar@{}[dr]|{PO} \;\ar@{>->}[r]^i & Y \ar[d]^{g'}\\
X' \;\ar@{>->}[r]_{i'} & Y'\\
}$$
where $d'$ is a deflation and $i'$ is an inflation, the left square being a pull-back and the right being a push-out.
\end{itemize}

Throughout this article, let $\B$ be a Krull-Schmidt exact category over a field $k$. Any subcategory we discuss in this article will be full and closed under isomorphisms, we assume $\B$ has enough projectives and enough injectives. Let $\mathcal P$ (resp. $\mathcal I$) be the subcategory of projective (resp. injective) objects. We denote by $\B/\D$ the category whose objects are objects of $\B$ and whose morphisms are elements of $\Hom_{\B}(A,B)/\D(A,B)$ for $A,B\in\B$, where $\D(A,B)$ is the subgroup of $\Hom_{\B}(A,B)$ consisting of morphisms which factor through objects in $\D$. Such category is called the quotient category of $\B$ by $\D$. \\

Cotorsion pair plays an important role in this article, the theory for cotorsion pairs on exact categories are developed in \cite{L,L2} and further in \cite{LN} on a more general structure called extriangulated category. Let's recall its definition on exact category.

\begin{defn}
Let $\U$ and $\V$ be two subcategories of $\B$ which are closed under direct summands. We call $(\U,\V)$ a \emph{cotorsion pair} if it satisfies the following conditions:
\begin{itemize}
\item[(a)] $\Ext^1_\B(\U,\V)=0$.

\item[(b)] For any object $B\in \B$, there exist two short exact sequence
\begin{align*}
V_B\rightarrowtail U_B\twoheadrightarrow B,\quad
B\rightarrowtail V^B\twoheadrightarrow U^B
\end{align*}
satisfying $U_B,U^B\in \U$ and $V_B,V^B\in \V$.
\end{itemize}
\end{defn}

A subcategory $\C$ is called \emph{fully rigid} if $\mathcal P\varsubsetneq \C$ is rigid, it admits a cotorsion pair $(\C,\K)$ and $\B/\K\simeq \mod(\C/\mathcal P)$ (see Definition \ref{fully} and Proposition \ref{eq} for more details, note that cluster tilting subcategory is a special case of fully rigid subcategory). \\
In this article we work on a cotorsion pair $(\C,\K)$ where $\C$ is fully rigid and the abelian quotient $\B/\K$ induced by it. If we consider a cluster tilting subcategory $\M$ on $\B$, we can get an ideal quotient $\B/\M\simeq \mod (\M/\mathcal P)$ which is abelian \cite[Theorem 3.2]{DL}, but this quotient category may not be $1$-Gorenstein anymore (see Example \ref{ex1}).\\ 

To introduce the main theorems, we need some notions.

\begin{defn}
Let $\B'$ and $\B''$ be two subcategories of $\B$.
\begin{itemize}
\item[(a)] Denote by $\CoCone(\B',\B'')$ the subcategory
$$\{ \text{ } X\in \B \text{ }|\text{ } \exists \text{ short exact sequence } X \rightarrowtail B' \twoheadrightarrow B'' \text{, }B'\in \B' \text{ and }B''\in \B'' \text{ } \};$$
\item[(b)] Denote by $\Cone(\B',\B'')$ the subcategory
$$\{ \text{ } X\in \B \text{ }|\text{ } \exists \text{ short exact sequence } B' \rightarrowtail B'' \twoheadrightarrow X \text{, }B'\in \B' \text{ and }B''\in \B'' \text{ } \};$$
\item[(c)] Let $\Omega^0 \B'=\B'$ and $\Omega \B'=\CoCone(\mathcal P,\B')$, then we can define $\Omega^i \B'$ inductively:
$$\Omega^i \B'=\CoCone(\mathcal P,\Omega^{i-1} \B'),$$
we can define a functor $\Omega: \B \to \B/\mathcal P$ according to the defintion above;
\item[(d)] Let $\Sigma^0 \B'=\B'$, $\Sigma \B'=\Cone(\B',\mathcal I)$, then we can define $\Sigma^i \B'$ inductively:
$$\Sigma^i \B'=\Cone(\Sigma^{i-1} \B',\mathcal I),$$
we can define a functor $\Sigma: \B \to \B/\mathcal I$ according to the defintion above.
\end{itemize}
We write an object $D$ in the form $\Omega B$ if it admits a short exact sequence $D \rightarrowtail P \twoheadrightarrow B$ where $P\in \mathcal P$. We write an object $D'$ in the form $\Sigma B'$ if it admits a short exact sequence $B' \rightarrowtail I \twoheadrightarrow D' $ where $I\in \mathcal I$.
\end{defn}

Let $\W$ and $\R$ be subcategories of $\B$, let
$$\W_\R=\{W\in \W \text{ }|\text{ }W \text{ has no direct summand in }\R \}\cup \{\text{zero objects of }\W\},$$
we give a sufficient-necessary condition for $\B/\K$ being $1$-Gorenstein (see Proposition \ref{eq}, Lemma \ref{D}, Proposition \ref{inj}, Theorem \ref{Gro} and Corollary \ref{Gor} for more details).

\begin{thm}
Let $(\C,\K),(\K,\D)$ be cotorsion pairs where $\C$ is fully rigid, then any object $X\in \B_\K$ admits the following short exact sequences:
$$X\rightarrowtail C\twoheadrightarrow C', \quad D'\rightarrowtail D\twoheadrightarrow X$$
where $C,C'\in \C$ and $D,D'\in \D$. $\B/\K$ is $1$-Gorenstein if and only if
\begin{itemize}
\item[(a)] for any object $X\in (\Sigma \D)_{\Omega \C}\cap \B_\K$, which is non-projective injective in $\B/\K$, for the short exact sequence $\xymatrix{X \ar@{ >->}[r]^x &C^1 \ar@{->>}[r] &C^2}$ where $C^1,C^2,\in \C$, morphism $\Omega x$ becomes zero in $\B/\K$.
\item[(b)] for any object $X\in (\Omega \C)_{\Sigma \D}\cap \B_\K$, which is non-injective projective in $\B/\K$, for the short exact sequence $\xymatrix{D_2 \ar@{ >->}[r] &D_1 \ar@{->>}[r]^y &Y}$
where $D_1,D_2,\in \D$, morphism $\Sigma y$ becomes zero in $\B/\K$.
\end{itemize}
\end{thm}


Cluster tiling subcategories are fully rigid, but this is not always true for $d$-cluster tilting subcategories $\N$ when $d\geq 3$ (see Example \ref{ex2}). By definition a $d$-cluster tilting subcategory $\C$ is rigid and admits a cotorsion pair $(\C,\K)$. 

In section 4 we prove the following theorem.

\begin{thm}
Let $\C$ be a $d$-cluster tilting subcategory of $\B$, $d\geq 3$. 

\begin{itemize}
\item[(a)] $\C$ is fully rigid if and only if $\Omega (\h_{\K})\subseteq \K$;
\item[(b)] If $\C$ is fully rigid, then $\B/\K$ is hereditary.
\end{itemize}
\end{thm}

In the last section we investigate the global dimension of $\B/\K$.

\section{preliminary}


In this section we first show some lemmas which will be used in the later sections. Then we give the definition of fully rigid subcategory and show a useful property of it.


\begin{lem}\label{ind}
Let $X$ be an indecomposable object.
\begin{itemize}
\item[(a)] In the short exact sequence $\xymatrix{\Omega X\ar@{ >->}[r]^q &P \ar@{->>}[r]^p &X}$ where $p$ is a minimal right $\mathcal P$-approximation, if $\Ext^1_\B(X, P)=0$, then $\Omega X$ is indecomposable;
\item[(b)] In the short exact sequence $\xymatrix{ X\ar@{ >->}[r]^i &I \ar@{->>}[r] &\Sigma X}$ where $i$ is a minimal left $\mathcal I$-approximation, if $\Ext^1_\B(I,X)=0$, then $\Sigma X$ is indecomposable.
\end{itemize}
\end{lem}

\begin{proof}
We only prove (a), (b) is by dual.\\
Let $\Omega X=Y_1\oplus Y_2$ where $Y_1\notin \mathcal P$ is indecomposable, then morphism $\Omega X\xrightarrow{q} P$ has the form $Y_1\oplus Y_2\xrightarrow{\svech{q_1}{q_2}} P$. Morphism $q_1$ admits a short exact sequence $\xymatrix{Y_1 \ar@{ >->}[r]^{q_1} &P \ar@{->>}[r]^{p_1} &X_1}$, since $\Ext^1_\B(X,P)=0$, we get the following commutative diagram:
$$\xymatrix{
Y_1 \ar[d]_{\svecv{1}{0}} \ar@{ >->}[r]^{q_1} &P \ar@{=}[d] \ar@{->>}[r]^{p_1} &X_1 \ar[d] \\
Y_1\oplus Y_2 \ar@{ >->}[r]^-{\svech{q_1}{q_2}} \ar[d]_-{\svech{1}{0}} &P \ar@{->>}[r]^p \ar[d]^a &X \ar[d]\\
Y_1 \ar@{ >->}[r]^{q_1} &P \ar@{->>}[r]^{p_1} &X_1.
}
$$
This implies $X_1$ is a direct summand of $P\oplus X$. If $X_1$ is a direct summand of $P$, then the first row splits and $Y_1\in \mathcal P$, a contradiction. Hence $X$ is a direct summand of $X_1$. Let $X_1=X\oplus X'$, we have the following commutative diagram
$$\xymatrix{
Y_1\oplus Y_2 \ar@{ >->}[r]^-{\svech{q_1}{q_2}} \ar[d]_{\beta} &P \ar@{->>}[r]^p \ar[d]^b &X \ar[d]^{\svecv{1}{0}}\\
Y_1 \ar[d]_{\alpha} \ar@{ >->}[r]^{q_1} &P \ar[d]^a \ar@{->>}[r]^{p_1} &X\oplus X' \ar[d]^{\svech{1}{0}} \\
Y_1\oplus Y_2 \ar@{ >->}[r]^-{\svech{q_1}{q_2}} &P \ar@{->>}[r]^{p} &X.
}
$$
Since $p$ is right minimal, we get $ab$ is an isomorphism. Then $\alpha\beta$ is also an isomorphism, which implies $\Omega X=Y_1$ is indecomposable.\\
\end{proof}

\begin{lem}\label{summand}
Let $\X$ be a subcategory which is closed under direct summands.
\begin{itemize}
\item[(a)] If $\Ext^1_\B(\X,\mathcal P)=0$ and $\mathcal P\subseteq \X$, then $\Omega \X$ is also closed under direct summands;
\item[(b)] If $\Ext^1_\B(\mathcal I,\X)=0$ and $\mathcal I\subseteq \X$, then $\Sigma \X$ is also closed under direct summands.
\end{itemize}
\end{lem}

\begin{proof}
We only prove (a), (b) is by dual.\\
Let $X\in \X$, it admits a short exact sequence $\xymatrix{\Omega X \ar@{ >->}[r]^{q} &P \ar@{->>}[r]^{p} &X}$ where $P\in \mathcal P$. Let $\Omega X=Y_1\oplus Y_2$, then morphism $\Omega X\xrightarrow{q} P$ has the form $Y_1\oplus Y_2\xrightarrow{\svech{q_1}{q_2}} P$. Morphism $q_1$ admits a short exact sequence $\xymatrix{Y_1 \ar@{ >->}[r]^{q_1} &P \ar@{->>}[r]^{p_1} &X_1}$, since $\Ext^1_\B(\X,\mathcal P)=0$, we get the following commutative diagram:
$$\xymatrix{
Y_1 \ar[d]_{\svecv{1}{0}} \ar@{ >->}[r]^{q_1} &P \ar@{=}[d] \ar@{->>}[r]^{p_1} &X_1 \ar[d]^{x_1} \\
Y_1\oplus Y_2 \ar@{ >->}[r]^-{\svech{q_1}{q_2}} \ar[d]_-{\svech{1}{0}} &P \ar@{->>}[r]^p \ar[d]^a &X \ar[d]^{x_1}\\
Y_1 \ar@{ >->}[r]^{q_1} &P \ar@{->>}[r]^{p_1} &X_1.
}
$$
Then $X_1$ is a direct summand of $X\oplus P\in \X$. Since $\X$ is closed under direct summands, we have $X_1\in \X$. Hence by definition $Y_1\in \Omega \X$.
\end{proof}


According to the definition of cotorsion pair, we get the following useful remark.

\begin{rem}
Let $(\U,\V)$ be a cotorsion pair of $\B$, then
\begin{itemize}
\item[(a)] $\U={^{\bot_1}}\V:=\{X \in \B\text{ }|\text{ } \Ext^1_\B(X,\V)=0\}$;

\item[(b)] $\V=\U^{\bot_1}:=\{ X \in \B \text{ }| \text{ }\Ext^1_\B(\U,X)=0\}$;

\item[(c)] $\U$ and $\V$ are closed under extension;

\item[(d)] $\mathcal P \subseteq \U$ and $\mathcal I \subseteq \V$.

\end{itemize}
\end{rem}

\begin{defn}
Let $\h$ be the subcategory consisting of objects $B\in \B$ which admits two short exact sequence
\begin{align*}
V_B\rightarrowtail U_B\twoheadrightarrow B,\quad
B\rightarrowtail V^B\twoheadrightarrow U^B
\end{align*}
where $U^B\in \U$ and $V_B\in \V$, $U_B,V^B\in \U\cap \V$, quotient category $\h/(\U\cap \V)$ is called the \emph{heart} of $(\U,\V)$.
\end{defn}

The name "heart" of cotorsion pair comes from "heart of $t$-structure", since on triangulated category, cotorsion pair is a generalization of $t$-structure (see \cite{N1}).

We have the following theorem.

\begin{thm}\cite{L,LN}
The heart of a cotorsion pair $(\U,\V)$ is abelian.
\end{thm}

Let $(\C,\K)$ be a cotorsion pair where $\C$ is rigid, then $\h=\CoCone(\C,\C)$, $\h\cap\K=\C$ and $\h/\C$ is the heart of $(\C,\K)$. Let $H$ be the half-exact functor defined in \cite{L2}, it is an additive functor and has the following properties (see \cite[Section 3, Section 4]{L2} for details):

\begin{itemize}
\item For any short exact sequence $\xymatrix{A \ar@{ >->}[r]^{f} &B \ar@{->>}[r]^{g} &C}$ in $\B$, the sequence $F(A)\xrightarrow{F(f)} F(B)\xrightarrow{F(g)} F(C)$ is exact in $\h/\C$;
\item If $X\in \h$, then $H(X)=X$;
\item $H(f)=0$ if and only if $f$ factors through $\K$.
\end{itemize}

By the same method as in Lemma \ref{summand}, we get the following lemma, the proof is left to the readers.

\begin{lem}\label{summand2}
Let $(\C,\K)$ be a cotorsion pair where $\C$ is rigid, then $\h=\CoCone(\C,\C)$ is closed under direct summands.
\end{lem}

Now we introduce the following concept:

\begin{defn}\label{fully}
A rigid subcategory $\C$ is called \emph{fully rigid} if
\begin{itemize}
\item it admits a cotorsion pair $(\C,\K)$, $\C_{\mathcal P}\neq 0$.
\item any indecomposable object in $\B$ either belongs to $\K$ or belongs to $\h$.
\end{itemize}
A cotorsion pair $(\C,\K)$ is called \emph{fully rigid} if $\C$ is fully rigid.
\end{defn}

Denote $\B'/\mathcal P$ by $\underline \B'$ if $\mathcal P\subseteq \B'\subseteq \B$, for any morphism $f\colon A\to B$ in $\B$, denote by $\underline{f}$ the image of $f$ under the natural quotient functor $\B\to \uB$. We denote $\B'/\K$ by $\overline \B'$ if $\K\subseteq \B'\subseteq \B$, for any morphism $g\colon A\to B$ in $\B$, we denote by $\overline{g}$ the image of $g$ under the natural quotient functor $\B\to \oB$. We have the following result.

\begin{prop}\label{eq}
If $(\C,\K)$ is fully rigid, then
\begin{itemize}
\item[(a)] $\oB\simeq \h/\C\simeq \mod\underline {\Omega \C}\simeq \mod\underline \C$;
\item[(b)] subcategory $\underline {\Omega \C}$ is the enough projectives of $\oB$.
\end{itemize}
\end{prop}

\begin{rem}
If $(\C,\K)$ is fully rigid, then $\B_\K=\h_\K$.
\end{rem}

\begin{proof}
(a) Since $H(\K)=0$, we have the following commutative diagram:
$$\xymatrix@C=0.9cm@R0.4cm{
\B \ar[rr]^H \ar[dr]_{\pi} &&\h/\C\\
&\oB \ar@{.>}[ur]_{\overline H}
}
$$
we show $\overline H$ is an equivalence.\\
By the property of the half-exact functor $H$, $\overline H$ is dense. $H(f)=0$ if and only if $f$ factors through $\K$, hence $\overline H$ is faithful. By the definition of fully rigid, for any indecomposable object $A\notin \K$, $H(A)=A$ in $\h/\C$, hence $\overline H$ is full. Thus $\overline H$ is an equivalence. Since $H$ is half-exact, the functor $\pi:\B\to \oB$ is also half-exact.\\
The equivalence $\oB\simeq \mod \underline {\Omega \C}$ is give by \cite[Theorem 1.2]{LZ}. \\
The equivalence $\h/\C\simeq \mod \underline \C$ is given by \cite[Theorem 3.2]{DL}.\\
Hence we have $\oB\simeq \h/\C\simeq \mod\underline {\Omega \C}\simeq \mod\underline \C$.\\
(b) Note that a morphism $f$ in $\Omega \C$ factors through $\K$ if and only if it factors through $\mathcal P$. Since $\oB\simeq \h/\C$ and $\underline {\Omega \C}$ is the enough projectives in $\h/\C$ by \cite[Theorem 4.10]{LN}, we get subcategory $\underline {\Omega \C}$ is the enough projectives of $\oB$.
\end{proof}

According to this proposition, the definition of fully rigid in this article is an analog of \cite[Definition 5.1]{B}.

\begin{lem}\label{D}
Let $(\C,\K),(\K,\D)$ be cotorsion pairs where $\C$ is fully rigid, then any object $X\in \h_\K$ belongs to $\Cone(\D,\D)$.
\end{lem}

\begin{proof}
Since $\D\subseteq \K^{\bot_1}\subseteq \C^{\bot_1}=\K$, we get $\D$ is rigid and the heart of $(\K,\D)$ is $\Cone(\D,\D)/\D$.\\
Let $X\in \h_\K$ be an indecomposable obejct, it admits the following commutative diagram
$$\xymatrix{
D_X \ar@{ >->}[r] \ar@{=}[d] &K_X \ar@{->>}[r] \ar@{ >->}[d] &X \ar@{ >->}[d]^x\\
D_X \ar@{ >->}[r] &D \ar@{->>}[r] \ar@{->>}[d] &Y \ar@{->>}[d]\\
&K \ar@{=}[r] &K
}
$$
where $D_X,D\in \D$ and $K_X,K\in \K$. Then $Y\in \Cone(\D,\D)$, by applying $H=\pi$ to this diagram, we get an isomorphism $X\xrightarrow{\overline x} Y$. Then $x$ is a section and $X$ is a direct summand of $Y$, by the dual of Lemma \ref{summand2}, $X\in \Cone(\D,\D)$.
\end{proof}

\section{Gorenstein property}

In this section, let $(\C,\K)$ be fully rigid. We also assume $\K$ admits another cotorsion pair $(\K,\D)$.

\begin{prop}\label{inj}
Subcategory $(\Sigma \D)/\mathcal I$ is the enough injectives in $\oB$.
\end{prop}

\begin{proof}
This is an analog of the results in \cite[Proposition 4.9,Theorem 4.10]{LN}.
\end{proof}

One typical kind of example of fully rigid subcategories is cluster tilting subcategories (please check the related results in \cite[Proposition 10.5]{L}). According to \cite{KZ}, $\mod \M$ is 1-Gorenstein if $\M$ is a cluster tilting subcategory on a triangulated category. In \cite{B}, a generalized result was shown for fully rigid subcategories on triangulated category (see \cite[Theorem 5.2]{B}). Unfortunately, such results are not always true on exact categories, even for cluster tilting subcategories. Here is an example.

\begin{exm}\label{ex1}
Let $\Lambda$ be the $k$-algebra given by the quiver
$$
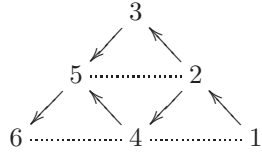
\xymatrix@C=0.4cm@R0.4cm{
&&3 \ar[dl]\\
&5 \ar[dl] \ar@{.}[rr] &&2 \ar[dl] \ar[ul]\\
6 \ar@{.}[rr] &&4 \ar[ul] \ar@{.}[rr] &&1 \ar[ul]}$$
with mesh relations. The AR-quiver of $\B:=\mod\Lambda$ is given by
$$
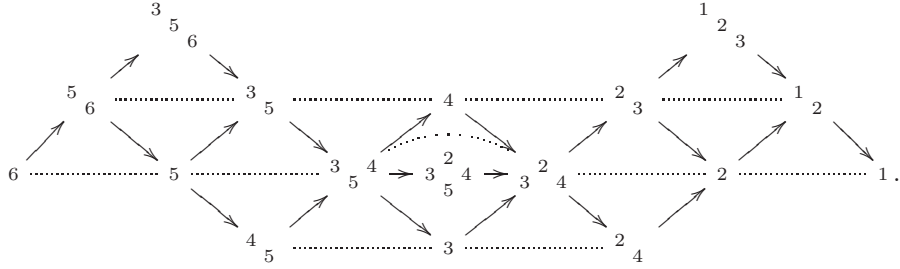
\xymatrix@C=0.3cm@R0.3cm{
&&{\begin{smallmatrix}
3&&\\
&5&\\
&&6
\end{smallmatrix}} \ar[dr] &&&&&&{\begin{smallmatrix}
1&&\\
&2&\\
&&3
\end{smallmatrix}} \ar[dr]\\
&{\begin{smallmatrix}
5&&\\
&6&
\end{smallmatrix}} \ar[ur] \ar@{.}[rr] \ar[dr] &&{\begin{smallmatrix}
3&&\\
&5&
\end{smallmatrix}} \ar@{.}[rr] \ar[dr] &&{\begin{smallmatrix}
4
\end{smallmatrix}} \ar@{.}[rr] \ar[dr] &&{\begin{smallmatrix}
2&&\\
&3&
\end{smallmatrix}} \ar[ur] \ar@{.}[rr] \ar[dr] &&{\begin{smallmatrix}
1&&\\
&2&
\end{smallmatrix}} \ar[dr]\\
{\begin{smallmatrix}
6
\end{smallmatrix}} \ar[ur] \ar@{.}[rr] &&{\begin{smallmatrix}
5
\end{smallmatrix}} \ar[ur] \ar@{.}[rr] \ar[dr] &&{\begin{smallmatrix}
3&&4\\
&5&
\end{smallmatrix}} \ar[ur] \ar[r] \ar[dr] \ar@{.}@/^15pt/[rr] &{\begin{smallmatrix}
&2&\\
3&&4\\
&5&
\end{smallmatrix}} \ar[r] &{\begin{smallmatrix}
&2&\\
3&&4
\end{smallmatrix}} \ar[ur] \ar@{.}[rr] \ar[dr] &&{\begin{smallmatrix}
2
\end{smallmatrix}} \ar[ur] \ar@{.}[rr] &&{\begin{smallmatrix}
1
\end{smallmatrix}}.\\
&&&{\begin{smallmatrix}
4&&\\
&5&
\end{smallmatrix}} \ar[ur] \ar@{.}[rr] &&{\begin{smallmatrix}
3
\end{smallmatrix}} \ar[ur] \ar@{.}[rr] &&{\begin{smallmatrix}
2&&\\
&4&
\end{smallmatrix}} \ar[ur]
}$$
We denote by ``$\circ$" in the AR-quiver the indecomposable objects belong to a subcategory. Let
$$
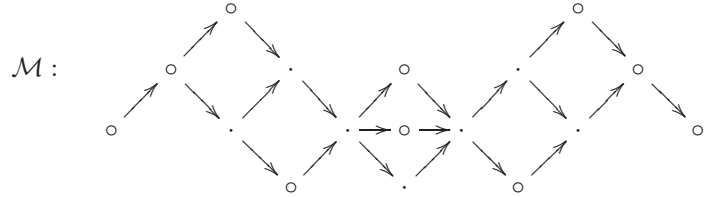
\xymatrix@C=0.4cm@R0.4cm{
&&&\circ \ar[dr] &&&&&&\circ \ar[dr]\\
{\M:} &&\circ \ar[ur]  \ar[dr] &&\cdot  \ar[dr] &&\circ  \ar[dr] &&\cdot  \ar[ur]  \ar[dr] &&\circ \ar[dr]\\
&\circ \ar[ur]  &&\cdot \ar[ur]  \ar[dr] &&\cdot \ar[ur] \ar[r] \ar[dr] &\circ \ar[r] &\cdot \ar[ur] \ar[dr] &&\cdot \ar[ur] &&\circ\\
&&&&\circ \ar[ur] &&\cdot \ar[ur] &&\circ \ar[ur]
}
$$
$\M$ is a cluster tilting subcategory of $\B$. Then $\B/\M\simeq \mod \underline {\Omega \M}$ and its quiver is the following:
 $$
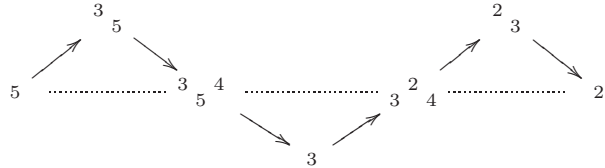
\xymatrix@C=0.4cm@R0.4cm{
&{\begin{smallmatrix}
3&\ \\
&5
\end{smallmatrix}} \ar[dr]
&&&&{\begin{smallmatrix}
2&\ \\
&3
\end{smallmatrix}}\ar[dr]\\
{\begin{smallmatrix}
\ &5&\
\end{smallmatrix}} \ar@{.}[rr] \ar[ur]
&&{\begin{smallmatrix}
3&&4\ \\
&5&
\end{smallmatrix}}\ar[dr] \ar@{.}[rr]
&&{\begin{smallmatrix}
&2&\ \\
3&&4
\end{smallmatrix}}\ar[ur] \ar@{.}[rr]
&&{\begin{smallmatrix}
\ &2&\
\end{smallmatrix}}\\
&&&{\begin{smallmatrix}
\ &3&\
\end{smallmatrix}} \ar[ur]}$$
It is not $1$-Gorenstein any more, the non-projective injective object ${\begin{smallmatrix}
 2
\end{smallmatrix}}$ have projective dimension $3$.
\end{exm}

We denote by $\pd_{\B}(X)$ the projective dimension of $X$ in $\B$ and by $\pd_{\oB}(X)$ (resp. $\id_{\oB}(X)$) the projective (resp. injective) dimension of $X$ in $\oB$, we prove the following theorem which will give an "if and only if" condition for $\oB$ being $1$-Gorenstein.

\begin{thm}\label{Gro}
Let $X\in \h_\K$, then
\begin{itemize}
\item[(a)] 
$\pd_{\oB}(X)\leq 1$ if and only if for the short exact sequence $\xymatrix{X \ar@{ >->}[r]^x &C^1 \ar@{->>}[r] &C^2}$ where $C^1,C^2,\in \C$, morphism $\overline{\Omega x}=0$.
\item[(b)] $\id_{\oB}(X)\leq 1$ if and only if for the short exact sequence $\xymatrix{D_2 \ar@{ >->}[r] &D_1 \ar@{->>}[r]^y &Y}$
where $D_1,D_2,\in \D$, morphism $\overline {\Sigma y}=0$.
\end{itemize}
\end{thm}

\begin{proof}
We only prove (a), (b) is by dual.\\
We have the following commutative diagram
$$\xymatrix{
&\Omega C^1 \ar@{=}[r] \ar@{ >->}[d]^-{\svecv{0}{1}} &\Omega C^1 \ar@{=}[r] \ar@{ >->}[d]^b &\Omega C^1\ar@{ >->}[d]^{p_1}\\
\Omega X \ar@{ >->}[r]^-{\svecv{p}{a}} \ar@{=}[d] &P_X\oplus \Omega C^1 \ar@{->>}[r]^-{\svech{-b'}{b}} \ar@{->>}[d]^-{\svech{1}{0}} &\Omega C^2 \ar@{ >->}[r]^{p_2} \ar@{->>}[d]^c &P\ar@{->>}[r]^{q_2} \ar@{->>}[d]^{q_1} &C^2 \ar@{=}[d] &({\maltese}) \\
\Omega X \ar@{ >->}[r]^p  &P_X \ar@{->>}[r]^q  &X \ar@{ >->}[r]^x  &C^1  \ar@{->>}[r]^d &C^2
}$$
It induces an exact sequence $\Omega X \xrightarrow{\overline a} \Omega C^1 \xrightarrow{\overline b} \Omega C^2 \xrightarrow{\overline c} X\to 0$ in $\oB$. Moreover, we have the following commutative diagram
$$\xymatrix{
\Omega X \ar[d]_{-a} \ar@{ >->}[r]^p &P_X \ar[d]^{-p_2b'} \ar@{->>}[r]^q &X \ar[d]^x\\
\Omega C \ar@{ >->}[r]_{p_1} &P \ar@{->>}[r]_{q_1} &C^1.
}
$$
It is not hard to check that $-\underline a=\Omega x$, then $-\overline a=\overline {\Omega x}$.
If $-\overline a=\overline {\Omega x}=0$, $X$ admits a short exact sequence $0\to \Omega C^1 \to \Omega C^2 \to X \to 0$ in $\oB$,
hence $\pd_{\oB}(X)\leq 1$.\\
Now we prove the "only if" part, The proof is divided into two steps.\\
(1) We show that in $({\maltese})$, $\underline b$ is right minimal.\\
We claim that $d$ is right minimal, otherwise $d$ can be written as $C^1_1\oplus C^1_2\xrightarrow{\svech{d_1}{0}} C^2$, this implies that $C^1_2$ is a direct summand of $X$. But $X\in \h_\K$, a contradiction. We claim that $\underline d$ is also right minimal. Since if we have a morphism $C^1\xrightarrow{c^1} C^1$ such that $\underline {dc^1}=\underline d$, then we have $d(c^1-1_{C^1}):C^1 \xrightarrow{p^1} P'\xrightarrow{p^2} C^2$ where $P'\in \mathcal P$. Morphism $p^2$ factors through $d$, we have $p^2:P'\xrightarrow{p^3} C^1\xrightarrow{d} C^2$, hence $d(c^1-1_{C^1})=dp^3p^1$, then $d(c^1-p^3p^1)=d$, which implies $c^1-p^3p^1$ is an isomorphism, hence $\underline {c^1}$ is an isomorphism. Hence $\underline b$ is right minimal. Let $b_1:\Omega C^1\to \Omega C^1$ be a morphism such that $\underline {bb_1}=\underline b$, then we have the following commutative diagram:
$$\xymatrix{
\Omega C^1 \ar@{ >->}[r]^{p_1} \ar[d]^{b_1} &P \ar@{->>}[r]^{q_1} \ar[d] &C^1 \ar[d]^{c_1}\\
\Omega C^1 \ar@{ >->}[r]^{p_1} \ar[d]^{b} &P \ar@{->>}[r]^{q_1} \ar@{=}[d] &C^1 \ar[d]^d\\
\Omega C^2 \ar@{ >->}[r]^{p_2}  &P \ar@{->>}[r]^{q_2} &C^2.
}
$$
Since $\underline {bb_1}=\underline b$, $bb_1-b$ factors through $p_1$, this implies $dc_1-d$ factors though $q_2$. Hence $\underline {dc_1}=\underline d$, which means $\underline {c_1}$ is an isomorphism. Let $\underline {c_2}$ be the inverse of $\underline {c_1}$, we have the following commutative diagram:
$$\xymatrix{
\Omega C^1 \ar@{ >->}[r]^{p_1} \ar[d]^{b_1} &P \ar@{->>}[r]^{q_1} \ar[d] &C^1 \ar[d]^{c_1}\\
\Omega C^1 \ar@{ >->}[r]^{p_1} \ar[d]^{b_2} &P \ar@{->>}[r]^{q_1} \ar[d] &C^1 \ar[d]^{c_2}\\
\Omega C^1 \ar@{ >->}[r]^{p_1}  &P \ar@{->>}[r]^{q_1} &C^1.
}
$$
Since $1_{C^1}-c_2c_1$ factors through $\mathcal P$, then it factors through $q_1$. Thus $1_{\Omega C^1}-b_2b_1$ factors through $p^1$, we have $\underline {b_2b_1}=\underline {1_{\Omega C^1}}$. By the similar argument we can find another morphism $b_2':\Omega C^1\to \Omega C^1$ such that $\underline {b_1b_2'}=\underline {1_{\Omega C^1}}$. Hence $\underline {b_1}$ is an isomorphism and $\underline b$ is right minimal.\\
(2) We show $\overline a=0$.\\
For any object $X\in  (\Sigma \D)_{\Omega \C}\cap \h_\K$, we have $\overline b$ and $\overline c$ are non-zero. Assume $\overline a\neq 0$, then we have the following exact sequence:
$$\xymatrix@C=0.5cm@R0.4cm{\Omega X \ar[rr]^{\overline a} \ar@{->>}[dr]_{\overline {r_4}} &&\Omega C^1 \ar[rr]^{\overline b} \ar@{->>}[dr]_{\overline {r_2}} &&\Omega C^2 \ar[r]^-{\overline c} &X \ar[r] &0\\
&R_2\ar@{ >->}[ur]_{\overline {r_3}}  &&R_1 \ar@{ >->}[ur]_{\overline {r_1}}}
$$
where $\Omega C^1 \xrightarrow{\overline {r_2}} R_1\xrightarrow{\overline {r_1}} \Omega C^2$ is an epic-monic factorization of $\overline b$ and $\Omega X \xrightarrow{\overline {r_4}} R_2\xrightarrow{\overline {r_3}} \Omega C^1$ is an epic-monic factorization of $\overline a$. Since $\oB$ is $1$-Gorenstein, $R_1\in \Omega \C$, hence we have a split short exact sequence $0\to R_2\xrightarrow{\overline {r_3}} \Omega C^1 \xrightarrow{\overline {r_2}} R_1\to 0$ which implies $R_2\in \Omega \C$. Thus $R_2$ is a direct summand of $\Omega X$. Then short exact sequence $\xymatrix{\Omega X \ar@{ >->}[r]^-{\svecv{a'}{a}} &P_X\oplus \Omega C^1 \ar@{->>}[r]^-{\svech{-b'}{b}} &\Omega C^2}$ has the following form
$$\xymatrix{S\oplus R_2\ar@{ >->}[r]^-{\left(\begin{smallmatrix}
a_{11}&a_{12}\\
a_{21}&a_{22}
\end{smallmatrix}\right)} &T\oplus R_2 \ar[r]^-{\svech{t}{r}} &\Omega C^2}$$
where $a_{22}$ is an isomorphism. Note that $\svech{\underline t}{\underline r}=\underline b$ is right minimal. We have an isomorphism $$T\oplus R_2 \xrightarrow{\left(\begin{smallmatrix}
1_T&a_{12}\\
0&a_{22}
\end{smallmatrix}\right)} T\oplus R_2$$ such that $\svech{t}{r}\left(\begin{smallmatrix}
1_T&a_{12}\\
0&a_{22}
\end{smallmatrix}\right)=\svech{t}{0}.$ But $\underline b$ is right minimal, this implies $R_2\in \mathcal P$, then $\overline a=0$, a contradiction. Hence $\overline a=0=\overline {\Omega x}$.

\end{proof}

We get the following corollary immediately.

\begin{cor}\label{Gor}
$\oB$ is $1$-Gorenstein if and only if
\begin{itemize}
\item for any object $X\in (\Sigma \D)_{\Omega \C}\cap \h_\K$ and short exact sequence $\xymatrix{X \ar@{ >->}[r]^x &C^1 \ar@{->>}[r] &C^2}$ where $C^1,C^2,\in \C$, morphism $\overline {\Omega x}=0$;
\item for any object $Y\in (\Omega \C)_{\Sigma \D}\cap \h_\K$ and short exact sequence $\xymatrix{D_2 \ar@{ >->}[r] &D_1 \ar@{->>}[r]^y &Y}$ where $D_1,D_2,\in \D$, morphism $\overline {\Sigma y}=0$.
\end{itemize}
\end{cor}

\begin{rem}
The following statement is also useful:
\begin{itemize}
\item If $\Omega (\Sigma \D)\subseteq \K$ and $\Sigma (\Omega \C)\subseteq \K$, then $\oB$ is $1$-Gorenstein.
\end{itemize}
The condition "$\Omega (\Sigma \D)\subseteq \K$ and $\Sigma (\Omega \C)\subseteq \K$" is very closed to what happens on triangulated category. It is automatically satisfied if $\B$ is Frobenius, since $\Omega (\Sigma \D)=\D\subseteq \K$ and $\Sigma (\Omega \C)=\C\subseteq \K$. In fact, it is a sufficient condition for $\oB$ being $1$-Gorenstein (a special case of Theorem \ref{Gro}), but not a necessary condition. Let's see an example.
\end{rem}

\begin{exm}
Let $\Lambda$ be the $k$-algebra given by the quiver
$$
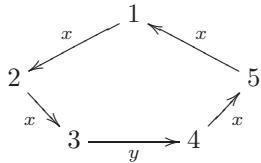
\xymatrix@C=0.4cm@R0.4cm{
&&1 \ar[dll]_x \\
2 \ar[dr]_x &&&&5 \ar[ull]_x\\
&3 \ar[rr]_y &&4 \ar[ur]_x
}
$$
with relations $x^4=0, x^3y=0,yx^3=0, xyx=0$, the AR-quiver of $\B:=\mod\Lambda$ is given by:
$$
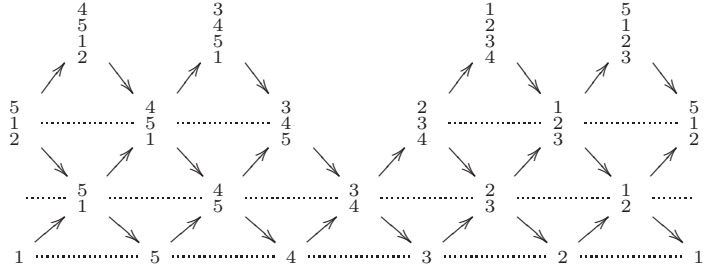
\xymatrix@C=0.3cm@R0.3cm{
&{\begin{smallmatrix}
4\ \\
5\ \\
1\ \\
2\
\end{smallmatrix}} \ar[dr]
&&{\begin{smallmatrix}
3\ \\
4\ \\
5\ \\
1\
\end{smallmatrix}} \ar[dr]
&&&& {\begin{smallmatrix}
1\ \\
2\ \\
3\ \\
4\
\end{smallmatrix}} \ar[dr]
&&{\begin{smallmatrix}
5\ \\
1\ \\
2\ \\
3\
\end{smallmatrix}} \ar[dr]\\
{\begin{smallmatrix}
5\ \\
1\ \\
2\
\end{smallmatrix}} \ar[dr] \ar[ur] \ar@{.}[rr]
&&{\begin{smallmatrix}
4\ \\
5\ \\
1\
\end{smallmatrix}} \ar[dr] \ar[ur] \ar@{.}[rr]
&&{\begin{smallmatrix}
3\ \\
4\ \\
5\
\end{smallmatrix}} \ar[dr]
&&{\begin{smallmatrix}
2\ \\
3\ \\
4\
\end{smallmatrix}} \ar[dr] \ar[ur] \ar@{.}[rr]
&&{\begin{smallmatrix}
1\ \\
2\ \\
3\
\end{smallmatrix}} \ar[dr] \ar[ur] \ar@{.}[rr]
&&{\begin{smallmatrix}
5\ \\
1\ \\
2\
\end{smallmatrix}}\\
\ar@{.}[r] &{\begin{smallmatrix}
5\ \\
1\
\end{smallmatrix}} \ar[dr] \ar[ur] \ar@{.}[rr]
&&{\begin{smallmatrix}
4\ \\
5\
\end{smallmatrix}} \ar[dr] \ar[ur] \ar@{.}[rr]
&&{\begin{smallmatrix}
3\ \\
4\
\end{smallmatrix}} \ar[dr] \ar[ur] \ar@{.}[rr]
&&{\begin{smallmatrix}
2\ \\
3\
\end{smallmatrix}} \ar[dr] \ar[ur] \ar@{.}[rr]
&&{\begin{smallmatrix}
1\ \\
2\
\end{smallmatrix}} \ar[dr] \ar[ur] \ar@{.}[r] &\\
{\begin{smallmatrix}
1
\end{smallmatrix}} \ar[ur] \ar@{.}[rr]
&&{\begin{smallmatrix}
5
\end{smallmatrix}} \ar[ur] \ar@{.}[rr]
&&{\begin{smallmatrix}
4
\end{smallmatrix}} \ar[ur] \ar@{.}[rr]
&&{\begin{smallmatrix}
3
\end{smallmatrix}} \ar[ur] \ar@{.}[rr]
&&{\begin{smallmatrix}
2
\end{smallmatrix}} \ar[ur] \ar@{.}[rr]
&&{\begin{smallmatrix}
1
\end{smallmatrix}}
}
$$
The first column and the last column are identical. We denote by ``$\circ$" in the AR-quiver the indecomposable objects belong to a subcategory. Let
$$
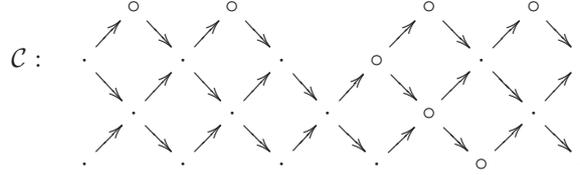
\xymatrix@C=0.3cm@R0.3cm{
&& \circ \ar[dr]
&& \circ \ar[dr]
&&&& \circ\ar[dr]
&& \circ \ar[dr]\\
\C: &\cdot \ar[dr] \ar[ur]
&&\cdot \ar[dr] \ar[ur]
&&\cdot \ar[dr]
&& \circ \ar[dr] \ar[ur]
&& \cdot \ar[dr] \ar[ur]
&& \cdot \\
&&\cdot \ar[dr] \ar[ur]
&& \cdot \ar[dr] \ar[ur]
&& \cdot \ar[dr] \ar[ur]
&& \circ \ar[dr] \ar[ur]
&& \cdot \ar[dr] \ar[ur]\\
&\cdot \ar[ur]
&& \cdot \ar[ur]
&& \cdot \ar[ur]
&& \cdot \ar[ur]
&& \circ \ar[ur]
&&\cdot
}
$$
$\C$ is a fully rigid subcategory of $\B$, since we have
$$
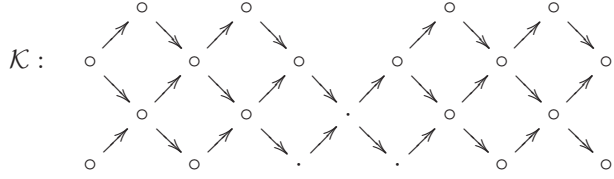
\xymatrix@C=0.3cm@R0.3cm{
&& \circ \ar[dr]
&& \circ \ar[dr]
&&&& \circ\ar[dr]
&& \circ \ar[dr]\\
\K: &\circ \ar[dr] \ar[ur]
&&\circ \ar[dr] \ar[ur]
&&\circ \ar[dr]
&& \circ \ar[dr] \ar[ur]
&& \circ \ar[dr] \ar[ur]
&& \circ \\
&&\circ \ar[dr] \ar[ur]
&& \circ\ar[dr] \ar[ur]
&& \cdot \ar[dr] \ar[ur]
&& \circ \ar[dr] \ar[ur]
&& \circ \ar[dr] \ar[ur]\\
&\circ \ar[ur]
&& \circ \ar[ur]
&& \cdot \ar[ur]
&& \cdot \ar[ur]
&& \circ \ar[ur]
&&\circ
}
$$
In this case, we have $\Omega (\Sigma \D)\subseteq \K$ and $\Sigma (\Omega \C)\subseteq \K$.
Let
$$
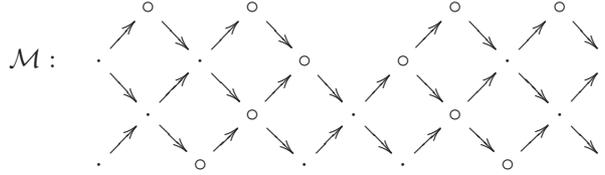
\xymatrix@C=0.3cm@R0.3cm{
&& \circ \ar[dr]
&& \circ \ar[dr]
&&&& \circ\ar[dr]
&& \circ \ar[dr]\\
\M: &\cdot \ar[dr] \ar[ur]
&&\cdot \ar[dr] \ar[ur]
&&\circ \ar[dr]
&& \circ \ar[dr] \ar[ur]
&& \cdot \ar[dr] \ar[ur]
&& \cdot \\
&&\cdot \ar[dr] \ar[ur]
&& \circ \ar[dr] \ar[ur]
&& \cdot \ar[dr] \ar[ur]
&& \circ \ar[dr] \ar[ur]
&& \cdot \ar[dr] \ar[ur]\\
&\cdot \ar[ur]
&& \circ \ar[ur]
&& \cdot \ar[ur]
&& \cdot \ar[ur]
&& \circ \ar[ur]
&&\cdot
}
$$
$\M$ is a cluster tilting subcategory and $\B/\M$ is $1$-Gorenstein, it satisfies the condition in Theorem \ref{Gro}, but we have the following short exact sequences:
$$
\xymatrix{ {\begin{smallmatrix}
4
\end{smallmatrix}} \ar@{ >->}[r] & {\begin{smallmatrix}
2\ \\
3\ \\
4\
\end{smallmatrix}} \ar@{->>}[r] &{\begin{smallmatrix}
2\ \\
3\
\end{smallmatrix}},\\
}
\quad
\xymatrix{ {\begin{smallmatrix}
4
\end{smallmatrix}} \ar@{ >->}[r] & {\begin{smallmatrix}
1\ \\
2\ \\
3\ \\
4\
\end{smallmatrix}} \ar@{->>}[r] &{\begin{smallmatrix}
1\ \\
2\ \\
3\
\end{smallmatrix}}
}$$
where $ {\begin{smallmatrix}
4
\end{smallmatrix}}$ is projective in $\B/\M$, ${\begin{smallmatrix}
2\ \\
3\ \\
4\
\end{smallmatrix}}$ is a non-injective projective object, ${\begin{smallmatrix}
2\ \\
3\
\end{smallmatrix}}\in \M$, ${\begin{smallmatrix}
1\ \\
2\ \\
3\ \\
4\
\end{smallmatrix}}$ is an injective object in $\B$ and ${\begin{smallmatrix}
1\ \\
2\ \\
3\
\end{smallmatrix}}\in \B_\M$.
\end{exm}

\section{fully rigid $d$-cluster tilting subcategories}

Cluster tiling subcategories are fully rigid, but this is not always true for $d$-cluster tilting subcategories when $d\geq 3$.

\begin{exm}\label{ex2}
This example is from \cite{V}.\\
Let $\Lambda$ be the algebra given by the following quiver with relations:
$$
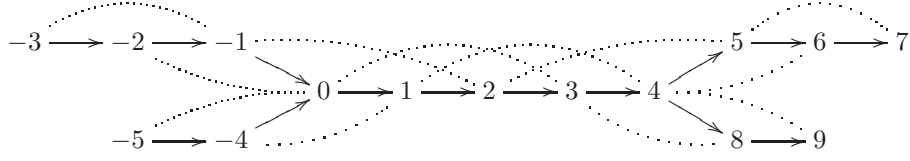
\xymatrix@C=0.7cm@R0.2cm{
&\\
-3 \ar[r] \ar@{.}@/^15pt/[rr] &-2 \ar@{.}@/_6pt/[drr]\ar[r] &-1 \ar@{.}@/^8pt/[drrr]\ar[dr] &&&&&&5 \ar[r] \ar@{.}@/^15pt/[rr] &6 \ar[r] &7\\
&&&0 \ar[r] \ar@{.}@/^18pt/[rrr] & 1 \ar[r] \ar@{.}@/^18pt/[rrr] &2 \ar[r] \ar@{.}@/^8pt/[urrr] & 3 \ar[r] \ar@{.}@/_8pt/[drr] &4 \ar[ur] \ar[dr] \ar@{.}@/_6pt/[urr] \ar@{.}@/^6pt/[drr] \\
&-5 \ar[r] \ar@{.}@/^6pt/[urr] &-4 \ar[ur] \ar@{.}@/_8pt/[urr] &&&&&&8 \ar[r] &9
}
$$
There is a $3$-cluster tilting subcategory $\C$ in $\B=\mod \Lambda$:
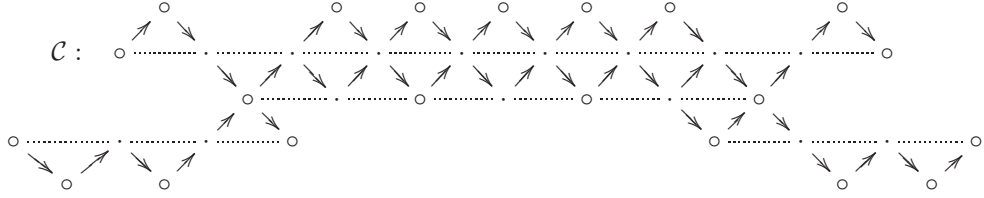
$$\xymatrix@C=0.2cm@R0.2cm{
&&&\circ \ar[dr] &&&&\circ \ar[dr] &&\circ \ar[dr] &&\circ \ar[dr] &&\circ \ar[dr] &&\circ \ar[dr] &&&&\circ \ar[dr]\\
&\C: &\circ \ar@{.}[rr] \ar[ur] &&\cdot \ar[dr] \ar@{.}[rr] &&\cdot \ar[dr] \ar[ur] \ar@{.}[rr] &&\cdot \ar[dr] \ar[ur] \ar@{.}[rr] &&\cdot \ar[dr] \ar[ur] \ar@{.}[rr] &&\cdot \ar[dr] \ar[ur] \ar@{.}[rr] &&\cdot \ar[dr] \ar[ur] \ar@{.}[rr] &&\cdot \ar[dr] \ar@{.}[rr] &&\cdot  \ar[ur] \ar@{.}[rr] &&\circ\\
&&&&&\circ \ar[dr] \ar[ur] \ar@{.}[rr] &&\cdot \ar[ur] \ar@{.}[rr] &&\circ \ar[ur] \ar@{.}[rr] &&\cdot \ar[ur] \ar@{.}[rr] &&\circ \ar[ur] \ar@{.}[rr] &&\cdot \ar[dr] \ar[ur] \ar@{.}[rr] &&\circ \ar[dr] \ar[ur]\\
\circ \ar@{.}[rr] \ar[dr] &&\cdot \ar@{.}[rr] \ar[dr] &&\cdot \ar[ur] \ar@{.}[rr] &&\circ &&&&&&&&&&\circ \ar[ur] \ar@{.}[rr] &&\cdot \ar[dr] \ar@{.}[rr] &&\cdot \ar[dr] \ar@{.}[rr] &&\circ\\
&\circ \ar[ur] &&\circ \ar[ur] &&&&&&&&&&&&&&&&\circ \ar[ur] && \circ \ar[ur]
}
$$
$\C$ is not fully rigid since we have
$$
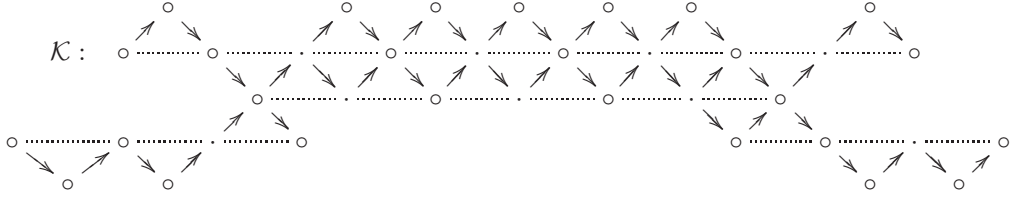
\xymatrix@C=0.2cm@R0.2cm{
&&&\circ \ar[dr] &&&&\circ \ar[dr] &&\circ \ar[dr] &&\circ \ar[dr] &&\circ \ar[dr] &&\circ \ar[dr] &&&&\circ \ar[dr]\\
&\K: &\circ \ar@{.}[rr] \ar[ur] &&\circ \ar[dr] \ar@{.}[rr] &&\cdot \ar[dr] \ar[ur] \ar@{.}[rr] &&\circ \ar[dr] \ar[ur] \ar@{.}[rr] &&\cdot \ar[dr] \ar[ur] \ar@{.}[rr] &&\circ \ar[dr] \ar[ur] \ar@{.}[rr] &&\cdot \ar[dr] \ar[ur] \ar@{.}[rr] &&\circ \ar[dr] \ar@{.}[rr] &&\cdot  \ar[ur] \ar@{.}[rr] &&\circ\\
&&&&&\circ \ar[dr] \ar[ur] \ar@{.}[rr] &&\cdot \ar[ur] \ar@{.}[rr] &&\circ \ar[ur] \ar@{.}[rr] &&\cdot \ar[ur] \ar@{.}[rr] &&\circ \ar[ur] \ar@{.}[rr] &&\cdot \ar[dr] \ar[ur] \ar@{.}[rr] &&\circ \ar[dr] \ar[ur]\\
\circ \ar@{.}[rr] \ar[dr] &&\circ \ar@{.}[rr] \ar[dr] &&\cdot \ar[ur] \ar@{.}[rr] &&\circ &&&&&&&&&&\circ \ar[ur] \ar@{.}[rr] &&\circ \ar[dr] \ar@{.}[rr] &&\cdot \ar[dr] \ar@{.}[rr] &&\circ\\
&\circ \ar[ur] &&\circ \ar[ur] &&&&&&&&&&&&&&&&\circ \ar[ur] && \circ \ar[ur]
}
$$
We can find that $\oB$ is not abelian.
\end{exm}

An natural question is when $d$-cluster tilting subcategories become fully rigid. In this section, we show the following theorem:

\begin{thm}\label{main}
Let $\C$ be a $d$-cluster tilting subcategory of $\B$, $d\geq 3$. $\C$ is fully rigid if and only if $\Omega (\h_{\K})\subseteq \K$.

\end{thm}

We will also figure out $\gl \oB$ when $\C$ is a $d$-cluster tilting subcategory and $\C$ is fully rigid.

We prove the following proposition, which is a general case of the "only if" part of Theorem \ref{main}.

\begin{prop}\label{onlyif}
If $(\C,\K)$ is fully rigid and $2$-rigid, then $\Omega(\h_\K)\subseteq \K$.
\end{prop}

\begin{proof}
Let $X\in \h_{\K}$ be an indecomposable object, it admits a short exact sequence $\xymatrix{\Omega X \ar@{ >->}[r]  &P_X \ar@{->>}[r]^p  &X}$ where $P_X\in \mathcal P$ and $p$ is right minimal. If $\Ext^1_{\B}(X,P_X)\neq 0$, we have a non-split short exact sequence $\xymatrix{P_X \ar@{ >->}[r]  &Y \ar@{->>}[r]^y  &X}$. Since $X\in \h$, it also admits a short exact sequence $\xymatrix{R^1 \ar@{ >->}[r] &R^1 \ar@{->>}[r]^r  &X}$ where $R^1,R^2\in \Omega \C$ and $0\neq \overline r$ is an epimorphism in $\oB$. Since $\C$ is $2$-rigid, we have $\Ext^1_{\B}(\Omega \C,\C)=0$, hence we have the following commutative diagram
$$\xymatrix{
R^1 \ar@{ >->}[r]  \ar[d] &R^1 \ar@{->>}[r]^r \ar[d] &X \ar@{=}[d] \\
P_X \ar@{ >->}[r]  &Y \ar@{->>}[r]^y  &X.
}
$$
We get $\overline y$ is both epic and monic in $\oB$, hence $\overline y$ is an isomorphism. Since $X$ is indecomposable, $y$ is a retraction, a contradiction. Hence $\Ext^1_\B(X,P_X)=0$, by Lemma \ref{ind}, $\Omega X$ is indecomposable. Since $\C$ is fully rigid, $\Omega X$ either belongs to $\K$ or belongs to $\h$. 
If $\Omega X\in \h_\K$, it admits a short exact sequence $\xymatrix{\Omega X \ar@{ >->}[r] &C_1 \ar@{->>}[r] &C_2}$ where $C_1,C_2\in \C$. Then we have the following commutative diagram
$$\xymatrix{
\Omega X \ar@{=}[d] \ar@{ >->}[r] &C_1 \ar[d] \ar@{->>}[r] &C_2 \ar[d]\\
\Omega X \ar@{ >->}[r]^-{\alpha} &P_X \ar@{->>}[r] &X
}
$$
which induces a short exact sequence $\xymatrix{C_1\ar@{ >->}[r] &C_2\oplus P_X \ar@{->>}[r] &X}$. This implies $X\in \C^{\bot_1}=\K$, a contradiction. 
Hence $\Omega X\in \K$.
\end{proof}

We get the following corollary immediately.

\begin{cor}
If $\C$ is fully rigid and $2$-rigid, then $\gl \oB\leq 1$.
\end{cor}

\begin{proof}
According to Proposition \ref{onlyif}, $\Omega (\h_\K)\subseteq \K$. Then every object $X\in \h_\K$ admits a short exact sequence $0 \to R_2\to R_1 \to X\to 0$ where $R_2, R_1\in \Omega \C$ in $\oB$. Hence $\gl \oB\leq 1$.
\end{proof}

Let ${^{\bot_1}}\C=\{X\in \B \text{ }|\text{ }\Ext^1_\B(X,\C)=0 \}$, the following lemma is useful in the prove of Theorem \ref{main}.

\begin{lem}\label{useful}
Let $\C$ be fully rigid, $K\in \K_\C\cap {^{\bot_1}}\C$ be an indecomposable object, then $\K$ satisfies one of the following conditions:
\begin{itemize}
\item[(a)] $K$ admits a short exact sequence $\xymatrix{K' \ar@{ >->}[r] &P \ar@{->>}[r]^p &K }$ where $P\in \mathcal P$, $K'\in \K$ and $p$ is right minimal;
\item[(b)] $K$ admits a short exact sequence $\xymatrix{C_2 \ar@{ >->}[r] &C_1 \ar@{->>}[r] &K}$ where $C_1,C_2\in \C$.
\end{itemize}

\end{lem}

\begin{proof}
$K$ always admits a short exact sequence $\xymatrix{X \ar@{ >->}[r] &P \ar@{->>}[r]^p &K }$ where $P\in \mathcal P$, $p$ is right minimal. By Lemma \ref{ind}, $X$ is indecomposable. Since $\C$ is fully rigid, we have $X\in \K$ or $X\in \h$.\\
If $X\in \K$, then (a) holds.\\
If $X\in \h_\K$, then $X$ admits a short exact sequence $\xymatrix{X \ar@{ >->}[r] &C^1 \ar@{->>}[r] &C^2 }$ where $C^1,C^2\in \C$. Then we have the following commutative diagram
$$\xymatrix{
X \ar@{ >->}[r] \ar@{=}[d] &C^1 \ar[d] \ar@{->>}[r] &C^2 \ar[d] \\
X \ar@{ >->}[r] &P \ar@{->>}[r]^p &K
}
$$
which induces a short exact sequence $\xymatrix{C^1 \ar@{ >->}[r] &C^2\oplus P \ar@{->>}[r] &K}$, hence (b) holds.
\end{proof}

Now let $\C$ be a $d$-cluster tilting subcategory, $d\geq 3$. Let $\C_1=\C$ and $\C_l$ ($2\leq l\leq d$) be the subcategory consisting of object $X$ admitting the following long exact sequence
$$\xymatrix@C=0.5cm@R0.4cm{
X \ar@{ >->}[r]^x &C^1 \ar@{->>}[dr] \ar[rr] &&C^2 \ar[r] &\cdot\cdot\cdot \ar[r] &C^{l-2} \ar@{->>}[dr] \ar[rr] &&C^{l-1} \ar@{->>}[r] &C^l\\
&&X^1 \ar@{ >->}[ur] &&&&X^l \ar@{ >->}[ur]
}
$$
where $C^i\in \C$, $x$ is a left $\C$-approximation and $X^1\in {^{\bot_1}}\C$. We have $\C_1\subseteq \C_2=\h\subseteq \cdot\cdot\cdot \subseteq \C_{d-1}\subseteq \C_{d}=\B$. Moreover, by the dual of \cite[Proposition 5.10]{LN}, $\C_l$ is closed under direct summands.\\

Now we show Theorem \ref{main}.

\begin{proof}
By Proposition \ref{onlyif}, we only need to show the "if" part of the theorem.

If $\Omega (\h_{\K})\subseteq \K$, we show that any indecomposable object $X\notin \K$ belongs to $\h_\K$.\\
If $X\in (\C_3)_\K$, then $X$ admits the a short exact sequence $\xymatrix{ X \ar@{ >->}[r] &C^1 \ar@{->>}[r] &X^1}$
where $C^1\in \C$ and $X^1\in \h$. If $X^1=X_1\oplus X_2$ such that $X_1\in \h_\K$ is non-zero, then we have the following commutative diagram:
$$\xymatrix{
\Omega X_1 \ar@{ >->}[r] \ar[d]_{x_1} &P \ar@{->>}[r]^p \ar[d]^{p_1} &X_1 \ar[d]^{\svecv{1}{0}}\\
X \ar@{ >->}[r] \ar[d]_{x_2} &C^1 \ar@{->>}[r] \ar[d]^{p_2} &X_1\oplus X_2 \ar[d]^{\svech{1}{0}}\\
\Omega X_1 \ar@{ >->}[r] &P \ar@{->>}[r]^p  &X_1
}
$$
where $P\in \mathcal P$ and $p$ is right minimal, hence $p_2p_1$ is an isomorphism and then $x_2x_1$ is also an isomorphism. Since $X$ is indecomposable, we get $X\simeq \Omega X_1$. But $\Omega X_1\in \K$, a contradiction. Hence $X^1\in \C$ and $X\in \h_\K$.\\
Assume we have shown that the statement holds for all objects in $(\C_l)_\K$. Let $X\in (\C_{l+1})_\K$, $X$ admits a short exact sequence $\xymatrix{ X \ar@{ >->}[r]^c &C^1 \ar@{->>}[r] &X^1}$ where $c$ is a left minimal $\C$-approximation and $X^1\in \C_l$. 
If $X^1$ has an indecomposable direct summand $X_1\notin \K$, then by the hypothesis, $X_1\in \h_\K$. Since $\Omega (\h_\K)\subseteq \K$, we have the following commutative digram
$$\xymatrix{
\Omega X_1 \ar@{ >->}[r] \ar[d]_{x_{11}} &P^1 \ar@{->>}[r]^{p^1} \ar[d]^{p_{11}} &X_1 \ar[d]^{\alpha}\\
X \ar@{ >->}[r] \ar[d]_{x_{21}} &C^1 \ar@{->>}[r] \ar[d]^{p_{21}} &X^1 \ar[d]^{\beta}\\
\Omega X_1 \ar@{ >->}[r] &P^1 \ar@{->>}[r]^{p^i}  &X_1
}
$$
where $P^1\in \mathcal P$, $p^1$ is minimal and $\beta\alpha=1_{X_1}$. We get $ X\simeq \Omega X_1\in \K$, a contradiction. 
Hence $X^1\in \K$. Since $X^1\in {^{\bot_1}}\C$, if $X^1$ has an indecomposable direct summand $X_2$ satisfying condition (a) in Lemma \ref{useful}, we have the following commutative diagram
$$\xymatrix{
\Omega X_2 \ar@{ >->}[r] \ar[d]_{x_{12}} &P^2 \ar@{->>}[r]^{p^1} \ar[d]^{p_{12}} &X_2 \ar[d]^{\alpha'}\\
X \ar@{ >->}[r] \ar[d]_{x_{22}} &C^1 \ar@{->>}[r] \ar[d]^{p_{22}} &X^1 \ar[d]^{\beta'}\\
\Omega X_2 \ar@{ >->}[r] &P^2 \ar@{->>}[r]^{p^i}  &X_2
}
$$
where $P^2\in \mathcal P$, $p^2$ is minimal and $\beta'\alpha'=1_{X_2}$. We get $ X\simeq \Omega \X_2\in \K$, a contradiction. Hence $X^1$ admits a short exact sequence $\xymatrix{C_2\ar@{ >->}[r] &C_1\ar@{->>}[r] &X^1}$ where $C_1,C_2\in \C$. Then we have the following commutative diagram
$$\xymatrix{
X \ar@{ >->}[r] \ar[d] &C^1 \ar@{->>}[r] \ar[d] &X^1 \ar@{=}[d]\\
C_2\ar@{ >->}[r] &C_1\ar@{->>}[r] &X^1
}
$$
which induces a short exact sequence $X\rightarrowtail C^1\oplus C_2 \twoheadrightarrow C_1$. Hence $X\in \h_\K$.

\end{proof}

\section{global dimension of $\oB$}

According to \cite[Theorem 5.2]{B}, the global dimension of the abelian quotient category induced by a fully rigid subcategory is either $1$ or infinity. This result is not always true on exact category (Example \ref{ex1}). In this section we investigate the global dimension of $\oB$.

\begin{lem}\label{syz}
Let $(\C,\K)$ be fully rigid, for an object $X\in \h_\K$, if $\pd_{\B} X=1$, then $\pd_{\oB} X\leq 1$.
\end{lem}

\begin{proof}
Since $\pd_{\B} X=1$, it admits the following commutative diagram.
$$\xymatrix{
&\Omega C^1 \ar@{=}[r] \ar@{ >->}[d] &\Omega C^1 \ar@{=}[r] \ar@{ >->}[d]^b &\Omega C^1\ar@{ >->}[d]\\
P_1 \ar@{ >->}[r] \ar@{=}[d] &P_0\oplus \Omega C^1 \ar@{->>}[r] \ar@{->>}[d] &\Omega C^2 \ar@{ >->}[r] \ar@{->>}[d]^c &P\ar@{->>}[r] \ar@{->>}[d] &C^2 \ar@{=}[d] \\
P_1 \ar@{ >->}[r]^p  &P_0 \ar@{->>}[r]^q  &X \ar@{ >->}[r]  &C^1  \ar@{->>}[r] &C^2
}$$
where $C^1,C^2,\in \C$, $P,P_0,P_1\in \mathcal P$. Hence $X$ admits a short exact sequence $0\to \Omega C^1 \to \Omega C^2 \to X\to 0$ in $\oB$, which means $\pd_{\oB}X\leq 1$.
\end{proof}

\begin{prop}\label{nsyz}
Let $(\C,\K)$ be fully rigid, $\gl \B=n$.
\begin{itemize}
\item[(a)] If $\Omega^2\C\subseteq \K$, then $\gl \oB\leq 2n-1$;
\item[(b)] If $n\geq 3$ and $\Omega^3\C\subseteq \K$, then $\gl \oB\leq 3n-1$.
\end{itemize}
\end{prop}

\begin{proof}
We have $n>1$, since if $n=1$, we get $\Omega \C=\mathcal P$, then by Proposition \ref{eq}, $\oB=0$. Hence $\K=\B$ and $\C=\mathcal P$, a contradiction to the definition of fully rigid subcategory.\\
Let $X\in \h$, $X$ admits the following commutative diagram.
$$\xymatrix{
&\Omega C^1 \ar@{=}[r] \ar@{ >->}[d] &\Omega C^1 \ar@{=}[r] \ar@{ >->}[d]^b &\Omega C^1\ar@{ >->}[d]\\
\Omega X \ar@{ >->}[r]^-{\svecv{a}{b}} \ar@{=}[d] &P_X\oplus \Omega C^1 \ar@{->>}[r] \ar@{->>}[d] &\Omega C^2 \ar@{ >->}[r] \ar@{->>}[d]^c &P\ar@{->>}[r] \ar@{->>}[d] &C^2 \ar@{=}[d] \\
\Omega X\ar@{ >->}[r]^p  &P_X \ar@{->>}[r]^q  &X \ar@{ >->}[r] &C^1 \ar@{->>}[r] &C^2
}$$
where $C^1,C^2\in \C$, $P_X,P\in \mathcal P$.\\
(a) If $\Omega^2\C\subseteq \K$, for any object $\Omega^{n-2}X\in \h_\K$, we have the following long exact sequence in $\oB$
$$0\to \Omega^{n-1}X \to \Omega C' \to \Omega C \to \Omega^{n-2} X \to 0 $$
where $C',C\in \C$. This implies $\pd_{\oB}(\Omega^{n-2} X)\leq 3$. In general, we have the following long exact sequence in $\oB$
$$0\to \Omega^l X \to \Omega C_l' \to \Omega C_l \to \Omega^{l-1} X \to 0 $$
where $C_l',C_l\in \C$ and $\pd_{\oB}(\Omega^l X)\leq 2(n-l)-1$, we get $\pd_{\oB}(\Omega^{l-1} X)\leq 2(n-l)+1$. At last we get $\pd_{\oB}X\leq 2n-1$. Hence $\gl \oB\leq 2n-1$.\\
(b) If $\Omega^3\C\subseteq \K$, for any object $C_0\in \C$, we have the following long exact sequence in $\oB$
$$0\to \Omega C' \to \Omega C \to \Omega^2 C_0 \to 0$$
where $C',C\in \C$, hence $\pd_{\oB}(\Omega^2 C_0)\leq 1$. Now for any object $\Omega^{n-2} X\in \h_\K$, we have the following long exact sequence in $\oB$
$$0\to \Omega^2 C_{n-2}' \to \Omega^2 C_{n-2} \to \Omega^{n-1}X \to \Omega C_{n-2}' \to \Omega C_{n-2} \to \Omega^{n-2} X \to 0 $$
where $C_{n-2}',C_{n-2}\in \C$. Hence $\pd_{\oB}(\Omega^{n-2} X)\leq 5$. In general, we have the following long exact sequence in $\oB$
$$0\to \Omega^2 C_{l-1}' \to \Omega^2 C_{l-1} \to \Omega^l X \to \Omega C_{l-1}' \to \Omega C_{l-1} \to \Omega^{l-1} X \to 0 $$
where $C_{l-1}',C_{l-1}\in \C$ and $\pd_{\oB}(\Omega^l X)\leq 3(n-l)-1$, then $\pd_{\oB}(\Omega^{l-1} X)\leq 3(n-l)+2$. At last we can get $\pd_{\oB}X\leq 3n-1$.
\end{proof}

We have the following corollary immediately.

\begin{cor}
Let $(\C,\K)$ be fully rigid.
\begin{itemize}
\item[(1)] If $\gl \B= 2$, then $\gl \oB\leq 3$;
\item[(2)] If $\gl \B= 3$, then $\gl \oB\leq 8$.
\end{itemize}
\end{cor}

When $\gl \B\geq 4$, we have the following proposition.

\begin{prop}
If $\gl \B=n,n\geq 4$, then $$\gl \oB< \infty \text{ if and only if }\sup\{\pd_{\oB}(\Omega^{i}C) \text{ }|\text{ } 2\leq i\leq n-2, \text{ } \forall C\in \C \}< \infty.$$
Moreover, if we know that $\sup\{\pd_{\oB}(\Omega^{i}C) \text{ }|\text{ } 2\leq i\leq n-2, \text{ } \forall C\in \C \}=m$, then $\gl \oB\leq m+3n-2$.
\end{prop}

\begin{proof}

The "only if" part is trivial. We check the "if" part.\\
$X$ admits the following long exact sequence in $\oB$:
$$0\to \Omega^{n-1}X \to \Omega^{n-1}C^1 \to \Omega^{n-1}C^2 \to \Omega^{n-2}X \to \cdot\cdot\cdot \to \Omega X \to \Omega C^1 \to \Omega C^2 \to X \to 0$$
where $C^1,C^2\in \C$. 
$\Omega^{n-2} X$ admits the following long exact sequence in $\oB$:
$$0\to \Omega^2C' \to \Omega^2C \to \Omega^{n-1}X \to \Omega C' \to \Omega C \to \Omega^{n-2} X \to 0  $$
where $C',C\in \C$. \\
Let $\sup\{\pd_{\oB}(\Omega^{i}C) \text{ }|\text{ } 2\leq i\leq n-2, \text{ } \forall C\in \C \}=m$, the case $m=0$ has been discussed in Proposition \ref{nsyz}, we assume $m>0$. Then $\pd_{\oB}(\Omega^2 C'),\pd_{\oB}(\Omega^2 C)\leq m$ and $\pd_{\oB}(\Omega^{n-1} X)\leq 1$, we get $\pd_{\oB}(\Omega^{n-2} X)\leq  m+4$. In general, $\Omega^{n-l} X$ ($2\leq l\leq n-1$) admits the following long exact sequence in $\oB$:
$$0\to \Omega^{l-1}C''' \to \Omega^{l-1}C'' \to \Omega^{n-1}X \to \cdot\cdot\cdot \to \Omega C''' \to \Omega C'' \to \Omega^{n-l} X \to 0  $$
where $C'',C'''\in \C$. Then $\pd_{\oB}(\Omega^{n-l} X)\leq m+3l-2$. At last we can get 
$\pd_{\oB}X\leq m+3n-2$.
\end{proof}

\end{document}